\documentclass[11pt,a4paper]{article}

\usepackage[utf8]{inputenc}
\usepackage[T1]{fontenc}
\usepackage[english]{babel}

\usepackage{lmodern}
\usepackage[margin=1in]{geometry}

\usepackage{amsmath}
\usepackage{amssymb}
\usepackage{amsthm}

\usepackage[hidelinks]{hyperref}

\theoremstyle{plain}
\newtheorem{theorem}{Theorem}[section]
\newtheorem{lemma}[theorem]{Lemma}

\theoremstyle{definition}
\newtheorem{definition}[theorem]{Definition}
\newtheorem{example}[theorem]{Example}

\theoremstyle{remark}
\newtheorem{remark}[theorem]{Remark}

\title{
Maximizing Measures for Subadditive and\\
Almost-Additive Potentials on Countable-Alphabet Subshifts
}

\author{
Ricardo Freire%
\thanks{
The first author's research was financed, in part,
by the S\~ao Paulo Research Foundation (FAPESP),
grant \#23/07076-4.
}
\\[0.4em]
{\small Department of Mathematics, University of S\~ao Paulo (IME-USP)}
\\
{\small Brazil}
\\
{\small \texttt{rfreire@usp.br}}
\and
Juliano S. Gonschorowski
\\[0.4em]
{\small Department of Mathematics -- DAEXT}
\\
{\small Federal Technological University of Paran\'a (UTFPR)}
\\
{\small Brazil}
\\
{\small \texttt{julianod@utfpr.edu.br}}
}

\date{}

\begin{document}

\maketitle

\begin{abstract}
We prove the existence of maximizing measures for subadditive and
almost-additive families of upper semicontinuous, bounded-above
potentials $\mathcal F=\{f_n\}_{n\geq1}$ on closed subshifts over
countable alphabets. Under a coercivity condition on $f_1$, we prove
that every maximizing sequence, that is,
$\{\mu_k\}_{k\geq1}\subset\mathcal M_\sigma(X)$ satisfying
$\bar{\mathcal F}(\mu_k)\to\beta(\mathcal F)$, is uniformly tight,
preventing mass from escaping to arbitrarily large symbols.
As a consequence, a maximizing measure exists. We apply these results 
to matrix cocycles, establishing the existence of invariant measures 
that realize the joint spectral radius for countable families of matrices.
\end{abstract}

\medskip


\medskip


\medskip

\section{Introduction}

Ergodic optimization studies the maximization of averages associated
with a dynamical system $\sigma:X\to X$ over the set of its invariant
probability measures $\mathcal M_\sigma(X)$. For a potential $f$, the fundamental
question is whether the value
\[
\sup_{\mu\in\mathcal M_\sigma(X)}\int_X f\,d\mu
\]
is attained by an invariant measure. For upper semicontinuous
potentials on compact spaces, existence follows from the weak-$*$
compactness of the space of invariant measures and the upper
semicontinuity of the integral functional. Foundational results in 
this area were established by Bousch \cite{bousch2001}, and a 
comprehensive account of the problems and methods can be found in 
the surveys by Jenkinson \cite{jenkinson2006, jenkinson2019}.

Here, we consider the shift map $\sigma$ on subshifts $X$ over
countable alphabets. In the noncompact case, sequences of measures may exhibit escape of mass, and it is not difficult to construct examples without any invariant maximizing measure (see example 2.1 in section 2.2). In this context, hypotheses such as coercivity of the potential and the BIP property for the shift, which somehow implies the good properties of compactness, are common in the area. See, for example \cite{sarig1999,sarig2003,iommiyayama2012, iommiyayama2014, pesin2014,iommitoddvelozo2022}.

Nonadditive sequences of potentials naturally arise in dimension
theory and the study of matrix products.
Kingman \cite{kingman1968} established the subadditive ergodic
theorem, providing a measure-theoretic foundation for the study
of asymptotic averages of subadditive families, without requiring
topological compactness.
In the compact setting, Falconer \cite{falconer1988} developed
a subadditive thermodynamic formalism for mixing repellers,
while King \cite{king1995} studied the multifractal structure
of general Sierpinski carpets.
Barreira \cite{barreira2006} further developed nonadditive
thermodynamic formalism for almost-additive sequences,
establishing a variational principle and studying the existence
and uniqueness of equilibrium and Gibbs measures under
appropriate additional hypotheses.
In the noncompact setting of countable Markov shifts,
Sarig \cite{sarig2003} studied the existence of Gibbs measures,
with the Big Images and Preimages (BIP) property playing
a central role.
Iommi and Yayama \cite{iommiyayama2012,iommiyayama2014}
developed the corresponding almost-additive theory and
studied zero-temperature limits in the BIP setting.

A classical motivation for subadditive families in ergodic optimization comes from matrix cocycles, where one seeks to maximize the top Lyapunov exponent. Following the foundational works of Berger and Wang \cite{bergerwang1992}, Bochi and Morris \cite{bochimorris2015}, and Feng \cite{feng2004}, it is known that the joint spectral radius of a finite set of matrices is closely related to the existence of maximizing measures. These classical results, however, inherently rely on the compactness of the underlying phase space (finite alphabets). In Section \ref{sec:matrices}, we demonstrate how our results apply to the noncompact setting of countable families of matrices. Our approach captures the asymptotic behavior of matrix products under minimal assumptions: we obtain invariant measures that realize the joint spectral radius without requiring the matrices to be positive, and crucially, without imposing any combinatorial constraints (such as the BIP property) on the underlying noncompact dynamics.

In  Bissacot and Freire \cite{bissacotfreire2014} the existence of maximizing measures for a coercive potential with bounded variation on an irreducible countable Markov shift is proved by means of cut and paste argument on the periodic orbits. In addition
to ensuring the existence of maximizing measures, the result also
provides information about the supports of these measures: they must be contained in a subshift over a finite alphabet. Since the difficulties in dealing with such argument are rather more complicated in the context of potential families, where the control of bounded variation for each member of the family is difficult, we take a step more in the abstract direction.

Garibaldi, Gomes, and Sobottka \cite{garibaldi2026} study the existence of maximizing measures for upper semicontinuous potentials on shifts over countable alphabets via \emph{blur shift spaces}. This compactification method is particularly interesting because it relaxes the strict coercivity requirement to a gap condition, assuming only that the asymptotic behavior at infinity is strictly less than the maximum ergodic average. However, to guarantee that the measure does not escape to the boundary of the compactified space, their approach typically still requires some combinatorial constraints on the shift (such as Bounded Preimages or the BIP property). Our theorem takes a complementary direction: by retaining the coercivity on the first potential, we are able to completely drop any combinatorial or Markovian assumptions on the shift space.

Almost-additive families on countable Markov shifts are also studied
by Iommi and Yayama \cite{iommiyayama2012} in the context of
thermodynamic formalism. Under combinatorial assumptions such as
BIP (Big Images and Preimages), these authors obtain results concerning 
Gibbs measures. Our objective is different: to establish a direct 
criterion for the existence of maximizing measures without invoking the BIP property.

Davie, Urba\'nski, and Zdunik~\cite{davieurbanskizdunik2007} studied the maximization of integrals of continuous potentials on metrizable noncompact spaces over prescribed closed families of probability measures. Their work provides an important noncompact framework and treats continuous potentials satisfying a coercivity condition.

Cuneo~\cite{cuneo2020} proved that, for a continuous
transformation $T\colon X\to X$ on a compact metric space, every
asymptotically additive sequence of potentials
$(f_n)_{n\geq1}\subset C(X)$ admits an additive representation with
uniformly sublinear error. This representation makes it possible to 
relate properties of these sequences to those of additive potentials.
The author also extended the representation result to topological spaces 
that are not necessarily compact. However, even in the 
almost-additive case, it is not clear that Cuneo's result \cite{cuneo2020} 
guarantees the existence of maximizing measures in the present settings, since 
extending this representation to noncompact spaces does not automatically 
yield the regularity required to find maximizing measures.

More recently, Cao \cite{cao2025} investigated equilibrium states and
zero-temperature limits for almost-additive potentials on topologically 
mixing countable Markov shifts without assuming the BIP property. His strategy 
finds maximizing measures as accumulation points of equilibrium 
states as the temperature tends to zero. However, there is a gap 
in the proof of his Lemma 3.2 (specifically equation 3.5). This gap 
does not enable us to conclude the existence of the equilibrium states whice are used in the zero-temperature limits to find maximizing measures. Thus, to the best of our knowledge, our work is one of the first in the subadditive and almost-additive settings to completely remove the requirement for the BIP property, which represents a strong combinatorial restriction in general.

Here, we consider the specific problem: the existence of measures maximizing 
the asymptotic averages of subadditive and almost-additive families. 
We consider families of potentials $\mathcal F := (f_n)_{n\geq1}$ defined 
on a closed subshift $X\subset\mathbb N^{\mathbb N}$. In the subadditive 
case, we assume that
\[
f_{n+m}(x)\leq f_n(x)+f_m(\sigma^n x),
\qquad n,m\geq1,
\]
and investigate the existence of measures attaining the maximal
asymptotic value
\[
\beta(\mathcal F)
:=
\sup_{\mu\in\mathcal M_\sigma(X)}
\lim_{n\to\infty}
\frac1n\int_X f_n\,d\mu.
\]

Our hypotheses are considerably weaker than usual. The key ingredient continues to be the coercivity of the first potential $f_1$:
\[
\lim_{i\to\infty} \left( \sup_{x\in[i]_X} f_1(x) \right) = -\infty.
\]

Coercivity is required only for the first member of the family, but it should be noted that this implies coercivity for all members. Also, we work with lower regularity requirements than usual, basically requiring only forward invariance on the subshift space. We only ask for closed subshifts that need not be Markov or irreducible, and replace the bounded variation assumption with upper semicontinuity and boundedness from above of the potentials.

% It is worth noting that the single-potential setting motivating our
% work naturally gives rise to the additive family
% $(S_nf)_{n\geq1}$, where
% $S_nf=\sum_{j=0}^{n-1}f\circ\sigma^j$ is the Birkhoff sum of $f$.
% This is the familiar additive case, which is included in both the
% subadditive and almost-additive settings.

Our main result is stated as follows.

\begin{theorem}\label{thm:main}
Let $X\subset\Sigma$ be a closed subshift. Let
$\mathcal F=(f_n)_{n\geq1}$ be a subadditive or almost-additive family
of upper semicontinuous functions that are bounded above.
Suppose that $f_1$ is coercive and that
\[
\beta(\mathcal F)>-\infty.
\]
Then there exists $\mu_\infty\in\mathcal M_\sigma(X)$ such that
\[
\bar{\mathcal F}(\mu_\infty)=\beta(\mathcal F).
\]
\end{theorem}

It is important to emphasize that our argument establishes existence
but does not locate the supports of maximizing measures within a
finite-alphabet subshift. In particular, additional assumptions on
the dynamics, such as transitivity of $\sigma$, and on the regularity
of the family appear to be necessary for such a conclusion.

The method developed in this article treats the asymptotic averages of subadditive 
and almost-additive families and establishes the existence of maximizing 
measures directly on the original subshift via a uniform tightness argument. 

% Our approach does not use zero-temperature limits. We provide a direct proof 
% for the existence of maximizing measures that does not require equilibrium states, 
% mixing conditions, or regularity assumptions beyond upper semicontinuity. 
% Using a uniform tightness argument based on the coercivity of $f_1$, we 
% obtain an existence result that applies to broader classes of subadditive 
% families.

Finally, recall that our theorem allows us to treat countable families of arbitrary matrices without requiring positivity or Markov assumptions, provided the norms decay to zero, bypassing the need for combinatorial properties like BIP entirely or finiteness of matrices.

In summary, our method requires only the invariance of the space, avoids 
strong combinatorial assumptions, and applies beyond the additive setting, with applications such as matrix cocycles.
Nevertheless, interesting questions remain, especially concerning 
the location of maximizing measures and the possibility of weakening the coercivity 
condition.

This article is organized as follows. In the next section, we give
the precise definitions and the setting of our results. In
Section~\ref{sec:tightness}, we prove the lemma containing our main
argument in the subadditive case. The main theorem is proved in
Section~\ref{sec:existence}. In Section \ref{sec:matrices} we present an application 
to products of matrices. Finally, in Section~\ref{sec:almost}, we treat the almost-additive 
case by reducing it to the subadditive case.

\medskip

\noindent\textbf{AI Declaration.} 
Artificial intelligence tools (ChatGPT model 6.1-Sol and Gemini model Pro 3.1) were used during the preparation of this manuscript to refine the English language, find relevant historical references for the introduction, to cross-check our mathematical arguments and, finally, to help structure the matrix cocycle application in Section~\ref{sec:matrices}. All AI-generated suggestions were independently verified, and the authors take full responsibility for the mathematical accuracy and the references provided in this work.

\section{Preliminaries}

\subsection{Configuration Space}
\label{subsec:espaco}

Let $\mathbb N=\{0,1,2,\dots\}$ be the set of natural numbers,
regarded as a countably infinite alphabet endowed with the discrete
topology, and let $\Sigma=\mathbb N^{\mathbb N}$ be the space of
sequences. We equip $\Sigma$ with the metric
\[
d(x,y):=
\begin{cases}
0, & \text{if }x=y,\\[2mm]
2^{-\min\{k\geq0:x_k\neq y_k\}}, & \text{if }x\neq y,
\end{cases}
\]
which induces the product topology. The space $\Sigma$ is Polish,
that is, complete and separable. On $\Sigma$, we consider the
classical \emph{shift} map $\sigma:\Sigma\to\Sigma$, defined by
$\sigma(x_0x_1x_2\dots)=(x_1x_2x_3\dots)$ for
$x=(x_0x_1x_2\dots)\in\Sigma$. Let $X\subset\Sigma$ be a closed
subshift, that is, a closed invariant subset satisfying
\[
\sigma(X)\subseteq X.
\]
Then the restriction $\sigma|_X\colon X\to X$ is continuous, and we
continue to denote it simply by $\sigma$. We denote by
$\mathcal M_\sigma(X)$ the set of $\sigma$-invariant Borel
probability measures on $X$. We adopt the convention that the
supremum of an empty set of values in the extended real line
is $-\infty$.

\subsection{Coercivity}

For each symbol $i\in\mathbb N$, we define the relative
first-level cylinder by
\[
[i]_X:=\{x\in X:x_0=i\}.
\]
The set $[i]_X$ may be empty; in this case, we adopt the
convention
\[
\sup_{x\in[i]_X}f(x)=-\infty.
\]
We say that a potential
$f\colon X\to\mathbb R\cup\{-\infty\}$ is \emph{coercive} if
\begin{equation}\label{eq:coerc}
\lim_{i\to\infty}\left(\sup_{x\in[i]_X}f(x)\right)=-\infty.
\end{equation}

\begin{example}[Escape of mass without coercivity]\label{ex:escape}
On the full shift $X=\mathbb N^{\mathbb N}$, define
\[
f(x):=-\frac{1}{x_0+1}
\quad\text{and}\quad
f_n(x):=\sum_{j=0}^{n-1}f(\sigma^j x).
\]
This is an additive family of continuous potentials bounded
above, but $f_1=f$ is not coercive.

For every invariant probability measure $\mu$,
\[
\bar{\mathcal F}(\mu)=\int f\,d\mu<0,
\]
since $f(x)<0$ at every point. On the other hand, if $\mu_i$
is the measure concentrated at the fixed point $(i,i,\ldots)$,
then
\[
\bar{\mathcal F}(\mu_i)=-\frac{1}{i+1}\longrightarrow0.
\]
Thus, $\beta(\mathcal F)=0$, but no maximizing measure exists.
The sequence $(\mu_i)_{i\geq 1}$ exhibits escape of mass.
\end{example}

\subsection{Subadditive Potentials}

\begin{definition}
A sequence of potentials $\mathcal F := (f_n)_{n\geq1}$, with
$f_n:X\to\mathbb R\cup\{-\infty\}$, is \emph{subadditive} if
$f_{n+m}\leq f_n+f_m\circ\sigma^n$ for all $n,m\geq 1$. We assume that each $f_n$
is upper semicontinuous (USC) and bounded above, with
$M_n:=\sup_{x\in X}f_n(x)<\infty$.
\end{definition}

For $\mu\in\mathcal M_\sigma(X)$, the subadditivity of
$\mathcal F$, the invariance of $\mu$, and the extended
Fekete lemma yield
\[
\lim_{n\to\infty}
\frac{1}{n}\int_Xf_n\,d\mu
=
\inf_{n\geq1}
\frac{1}{n}\int_Xf_n\,d\mu
\in[-\infty,\infty).
\]
We denote this value by $\bar{\mathcal F}(\mu)$ and define
\[
\beta(\mathcal F)
:=
\sup_{\mu\in\mathcal M_\sigma(X)}
\bar{\mathcal F}(\mu).
\]
In the subadditive case, we explicitly assume that
\[
\beta(\mathcal F)>-\infty.
\]
This integrability condition guarantees, in particular, that
$\mathcal M_\sigma(X)\neq\varnothing$.

\section{Uniform Tightness}\label{sec:tightness}

The preceding example shows that, without a condition controlling
large symbols, a maximizing sequence may exhibit escape of mass.
The following lemma shows that the coercivity of $f_1$ provides
this control in the subadditive case.

\begin{lemma}[Uniform tightness in the subadditive case]
\label{lem:tightness}
Let $\mathcal F=(f_n)_{n\geq1}$ be a subadditive family of
upper semicontinuous functions bounded above, and suppose that
$f_1$ is coercive. If
$(\mu_k)_{k\geq 1}\subset\mathcal M_\sigma(X)$ satisfies
\[
\liminf_{k\to\infty}\bar{\mathcal F}(\mu_k)>-\infty,
\]
then the sequence $(\mu_k)_{k\geq 1}$ is uniformly tight.
\end{lemma}

\begin{proof}
Let $L\in\mathbb R$ satisfy
$L<\liminf_{k\to\infty}\bar{\mathcal F}(\mu_k)$.
By the definition of the limit inferior, there exists an
index $N_0\geq1$ such that
$\bar{\mathcal F}(\mu_k)>L$ for every $k\geq N_0$.

Since
\[
\bar{\mathcal F}(\mu_k)
=
\inf_{n\geq1}\frac{1}{n}\int_Xf_n\,d\mu_k,
\]
the term corresponding to $n=1$ yields
\[
\int_Xf_1\,d\mu_k
\geq
\bar{\mathcal F}(\mu_k)
>
L.
\]

Therefore, since $f_1$ is bounded above and the integral is
strictly greater than $L$, the integral is a real number
for these measures.

Since
\[
\int_Xf_1\,d\mu_k>L
\]
for $k\geq N_0$, the function $f_1$ is not identically
$-\infty$. As it is bounded above, it follows that
\[
M:=\sup_{x\in X}f_1(x)\in\mathbb R.
\]
Define
\[
K_0:=|M|+|L|+1>0.
\]

Given $\varepsilon>0$, the coercivity condition
\eqref{eq:coerc} on $f_1$ ensures that, for each
$j\in\mathbb N$, there exists a finite set of symbols
$F_j\subset\mathbb N$ such that, for every
$x\in B_j:=\bigcup_{i\notin F_j}[i]_X$, we have
\[
f_1(x)\leq-C_j,
\qquad
\text{where }
C_j:=\frac{2^{j+1}K_0}{\varepsilon}+|M|+1.
\]

Define $p_k(j):=\mu_k(B_j)$. Decomposing the domain of
integration and using the fact that $f_1\leq M$ on $X$,
we obtain
\begin{align*}
\int_Xf_1\,d\mu_k
&=
\int_{X\setminus B_j}f_1\,d\mu_k
+
\int_{B_j}f_1\,d\mu_k\\
&\leq
M(1-p_k(j))-C_jp_k(j)
=
M-p_k(j)(M+C_j).
\end{align*}
Since $\int_Xf_1\,d\mu_k>L$, we deduce that
$L<M-p_k(j)(M+C_j)$. Note that the numerator satisfies
$M-L\leq|M|+|L|<|M|+|L|+1=K_0$. Furthermore,
$M+C_j=M+|M|+\frac{2^{j+1}K_0}{\varepsilon}+1
\geq\frac{2^{j+1}K_0}{\varepsilon}+1$.
Solving for $p_k(j)$, we obtain
\[
p_k(j)
<
\frac{M-L}{M+C_j}
\leq
\frac{|M|+|L|}
{\frac{2^{j+1}K_0}{\varepsilon}+1}
<
\frac{K_0}{\frac{2^{j+1}K_0}{\varepsilon}}
=
\frac{\varepsilon}{2^{j+1}}.
\]

Since $(\sigma^n x)_0=x_n$, we have
\[
A_n
:=
\{x\in X:x_n\notin F_n\}
=
\sigma^{-n}(B_n).
\]
By the invariance of $\mu_k$,
\[
\mu_k(A_n)
=
\mu_k\bigl(\sigma^{-n}(B_n)\bigr)
=
\mu_k(B_n)
<
\frac{\varepsilon}{2^{n+1}}.
\]

Define the set
\[
K_\varepsilon
:=
\{x\in X:x_n\in F_n\text{ for every }n\geq0\}
=
X\cap\prod_{n=0}^{\infty}F_n.
\]
Since each factor $F_n$ is finite,
$\prod_{n=0}^{\infty}F_n$ is compact in $\Sigma$ by
Tychonoff's theorem. As $X$ is closed, $K_\varepsilon$
is compact in $X$. We estimate the measure of its complement:
\[
\mu_k((K_\varepsilon)^c)
\leq
\sum_{n=0}^{\infty}\mu_k(A_n)
<
\sum_{n=0}^{\infty}\frac{\varepsilon}{2^{n+1}}
=
\varepsilon.
\]
Consequently, $\mu_k(K_\varepsilon)>1-\varepsilon$ for every
$k\geq N_0$.

Since $X$ is Polish, every Borel probability measure on $X$
is tight. Therefore, for each $1\leq k<N_0$, there exists
a compact set $C_k\subset X$ such that
\[
\mu_k(C_k)>1-\varepsilon.
\]

Defining the compact set
$K:=K_\varepsilon\cup C_1\cup\dots\cup C_{N_0-1}$,
we obtain $\mu_k(K)>1-\varepsilon$ for every $k\geq1$,
which establishes the uniform tightness of the sequence.
\end{proof}

\section{Existence of a Maximizing Measure}\label{sec:existence}

\begin{theorem}\label{thm:subadditive}
Let $X\subset\Sigma$ be a closed subshift. Let
$\mathcal F=(f_n)_{n\geq1}$ be a subadditive family of
upper semicontinuous functions bounded above.
Suppose that $f_1$ is coercive and that
\[
\beta(\mathcal F)>-\infty.
\]
Then there exists $\mu_\infty\in\mathcal M_\sigma(X)$ such that
\[
\bar{\mathcal F}(\mu_\infty)=\beta(\mathcal F).
\]
\end{theorem}

\begin{proof}
By the definition of the supremum
$\beta(\mathcal F)=
\sup_{\mu\in\mathcal M_\sigma(X)}\bar{\mathcal F}(\mu)$,
there exists a maximizing sequence
$(\mu_k)_{k\geq1}\subset\mathcal M_\sigma(X)$ satisfying
$\bar{\mathcal F}(\mu_k)\to\beta(\mathcal F)$.

Since this sequence satisfies the limit condition in
Lemma~\ref{lem:tightness}, it is uniformly tight. 
By Prohorov's theorem, uniform tightness gives us a subsequence 
$(\mu_{k_j})_{j\geq 1}$ converging in the weak-$*$ topology to a probability 
measure $\mu_\infty$.

To prove the invariance of the limiting measure, let
$\varphi\in C_b(X)$. Since $\sigma\colon X\to X$ is
continuous, we have $\varphi\circ\sigma\in C_b(X)$.
By the invariance of each $\mu_{k_j}$,
\[
\int_X\varphi\circ\sigma\,d\mu_{k_j}
=
\int_X\varphi\,d\mu_{k_j}.
\]
Passing to the limit using weak convergence of the measures,
we obtain
\[
\int_X\varphi\circ\sigma\,d\mu_\infty
=
\int_X\varphi\,d\mu_\infty.
\]
Since this equality holds for every $\varphi\in C_b(X)$,
the measure $\mu_\infty$ is $\sigma$-invariant. Therefore,
\[
\mu_\infty\in\mathcal M_\sigma(X).
\]
For each $n\geq1$, define
\[
\Phi_n(\mu)
:=
\frac{1}{n}\int_Xf_n\,d\mu.
\]
Since $f_n$ is upper semicontinuous and bounded above, the
extended version of the Portmanteau theorem implies, via the monotone convergence theorem, that the
functional $\Phi_n(\mu)$
is upper semicontinuous in the topology of weak convergence.
In particular,
\[
\Phi_n(\mu_\infty)
\geq
\limsup_{j\to\infty}\Phi_n(\mu_{k_j}).
\]

Since
\[
\bar{\mathcal F}(\mu)
=
\inf_{n\geq1}\Phi_n(\mu),
\]
and the infimum of a family of upper semicontinuous functions
is upper semicontinuous, we conclude that
$\bar{\mathcal F}$ is upper semicontinuous. Hence,
\[
\bar{\mathcal F}(\mu_\infty)
\geq
\limsup_{j\to\infty}
\bar{\mathcal F}(\mu_{k_j})
=
\beta(\mathcal F).
\]
On the other hand, by the definition of $\beta(\mathcal F)$,
\[
\bar{\mathcal F}(\mu_\infty)
\leq
\beta(\mathcal F).
\]
Therefore,
\[
\bar{\mathcal F}(\mu_\infty)
=
\beta(\mathcal F).
\]
\end{proof}

\section{Application to products of matrices}\label{sec:matrices}

% The study of Lyapunov exponents and the joint spectral radius for matrix cocycles is a cornerstone of ergodic optimization and dynamical systems, with foundational contributions by Berger and Wang \cite{bergerwang1992}, Bochi and Morris \cite{bochimorris2015}, and Feng \cite{feng2004}, among others. Locally constant linear cocycles provide the prototypical example of subadditive families. See, e.g., \cite{morris2010} for finite alphabets and \cite{iommiyayama2012} for the almost-additive positive-matrix case on countable Markov shifts. 

% Theorem~\ref{thm:subadditive} advances the current literature by allowing us to treat countable families of arbitrary matrices without requiring positivity or Markov assumptions, provided the norms decay to zero. This bypasses the need for combinatorial properties like BIP entirely.

Fix $d \ge 1$ and a sequence $(A_i)_{i\in\mathbb{N}}$ of real $d\times d$
matrices. Let $\|\cdot\|$ be a submultiplicative matrix norm and assume that
\begin{equation}\label{eq:matrix-coercive}
  \lim_{i\to\infty} \|A_i\| = 0 .
\end{equation}
Let $X\subset\Sigma$ be a closed subshift. For $x\in X$ and $n\ge 1$, define
\[
  f_n(x) := \log \bigl\| A_{x_{n-1}} A_{x_{n-2}} \cdots A_{x_0} \bigr\| ,
\]
with the convention $\log 0 = -\infty$, and let
$\mathcal F=(f_n)_{n\ge1}$.

For $n,m\ge1$ we have
\[
  A_{x_{n+m-1}}\cdots A_{x_0}
  = \bigl(A_{x_{n+m-1}}\cdots A_{x_n}\bigr)\bigl(A_{x_{n-1}}\cdots A_{x_0}\bigr),
\]
and the first factor on the right-hand side is the product defining
$f_m(\sigma^n x)$. By submultiplicativity of the norm,
\[
  f_{n+m}(x) \le f_n(x) + f_m(\sigma^n x).
\]

Each $f_n$ depends only on the coordinates
$x_0,\dots,x_{n-1}$, hence is locally constant. Therefore $f_n$ is continuous
as a map into $[-\infty,\infty)$ and, in particular, upper semicontinuous.
By \eqref{eq:matrix-coercive}, $\Lambda := \sup_{i}\|A_i\| < \infty$, so
$f_n \le n\log\Lambda$ and each $f_n$ is bounded above.

Since $f_1(x) = \log\|A_{x_0}\|$, we obtain
\[
  \sup_{x\in[i]_X} f_1(x) \le \log\|A_i\| \longrightarrow -\infty
  \quad\text{as } i\to\infty ,
\]
so $f_1$ is coercive.

Suppose that $X$ contains a
periodic point $x = (www\cdots)$, where $w = w_0\cdots w_{p-1}$ is a finite
word, and that the matrix $A_w := A_{w_{p-1}}\cdots A_{w_0}$ has spectral
radius $\rho(A_w)>0$. Let $\mu_w\in\mathcal M_\sigma(X)$ be the invariant
probability measure equidistributed on the orbit of $x$. Since the limit
defining $\bar{\mathcal F}(\mu_w)$ exists, it may be computed along the
subsequence $n=kp$, and Gelfand's formula gives
\[
  \bar{\mathcal F}(\mu_w)
  = \lim_{k\to\infty} \frac{1}{kp}\log\|A_w^{k}\|
  = \frac{1}{p}\log\rho(A_w) > -\infty .
\]
Hence $\beta(\mathcal F) > -\infty$.

By Theorem~\ref{thm:subadditive} there exists
$\mu_\infty\in\mathcal M_\sigma(X)$ such that
\[
  \bar{\mathcal F}(\mu_\infty) = \beta(\mathcal F)
  = \sup_{\mu\in\mathcal M_\sigma(X)} \lim_{n\to\infty} \frac{1}{n}
    \int_X \log\bigl\|A_{x_{n-1}}\cdots A_{x_0}\bigr\|\,d\mu(x),
\]
that is, the maximal top Lyapunov exponent of the matrix cocycle over $X$ is
attained by an invariant measure.

Let $X = \Sigma$ and
$\mathcal{A} = \{A_i\}_{i\in\mathbb{N}}$, which is a bounded set of matrices.
Its joint spectral radius is
\[
  \hat\rho(\mathcal{A})
  := \lim_{n\to\infty} \sup_{i_0,\dots,i_{n-1}\in\mathbb{N}}
    \bigl\|A_{i_{n-1}}\cdots A_{i_0}\bigr\|^{1/n}.
\]
For every $\mu\in\mathcal M_\sigma(\Sigma)$ and every $n\ge1$, we have
$\bar{\mathcal F}(\mu)\le \frac1n\int f_n\,d\mu \le \frac1n\sup f_n$, whence
$\beta(\mathcal F)\le\log\hat\rho(\mathcal{A})$. Conversely, the measures
$\mu_w$ above give
$$\beta(\mathcal F)\ge \sup_{w}\frac{1}{|w|}\log\rho(A_w),$$ and by the
Berger--Wang theorem~\cite{bergerwang1992}, valid for bounded sets of
matrices, the right-hand side equals $\log\hat\rho(\mathcal{A})$. Thus
\[
  \beta(\mathcal F) = \log\hat\rho(\mathcal{A}),
\]
and Theorem~\ref{thm:subadditive} shows that the joint spectral radius of a
countable family satisfying \eqref{eq:matrix-coercive} is realized by an
invariant measure on $\Sigma$. For finite families this is known; see
\cite{morris2010}.

Coercivity is still key here: let $d=1$ and
$A_i := 1-\frac{1}{i+2}$ for $i\in\mathbb{N}$. Then
$\hat\rho(\mathcal{A}) = 1$, so $\beta(\mathcal F)=0$, while
$\bar{\mathcal F}(\mu) = \int_\Sigma \log A_{x_0}\,d\mu < 0$ for every
$\mu\in\mathcal M_\sigma(\Sigma)$, and we're lacking coercivity. Hence no maximizing measure exists, as in Example~\ref{ex:escape}.

\section{The Almost-Additive Case}\label{sec:almost}

\begin{definition}\label{def:almost}
A sequence of potentials $(f_n)_{n\geq1}$, with
$f_n\colon X\to\mathbb R$, is called \emph{almost-additive} if there
exists a constant $C\geq0$ such that, for every $n,m\geq1$
and every $x\in X$,
\[
\left|f_{n+m}(x)-f_n(x)-f_m(\sigma^n x)\right|\leq C.
\]
\end{definition}

Now let $\mu\in\mathcal M_\sigma(X)$, and suppose that each
$f_n$ is bounded above. Then the integral
\[
a_n(\mu):=\int_Xf_n\,d\mu
\]
takes values in $[-\infty,\infty)$. Define
\[
g_n:=f_n+C,\qquad n\geq1.
\]
By almost-additivity,
\[
g_{n+m}
=f_{n+m}+C
\leq f_n+f_m\circ\sigma^n+2C
=g_n+g_m\circ\sigma^n.
\]
Therefore, $\mathcal G=(g_n)_{n\geq1}$ is subadditive.
The invariance of $\mu$ implies that the numerical sequence
\[
b_n(\mu):=\int_Xg_n\,d\mu=a_n(\mu)+C
\]
is also subadditive. By the extended version of Fekete's
lemma,
\[
\lim_{n\to\infty}\frac{b_n(\mu)}{n}
=
\inf_{n\geq1}\frac{b_n(\mu)}{n}
\in[-\infty,\infty).
\]
Since $b_n(\mu)=a_n(\mu)+C$ and $C/n\to0$, it follows
that the limit
\[
\bar{\mathcal F}(\mu)
:=
\lim_{n\to\infty}\frac1n\int_Xf_n\,d\mu
\]
also exists. More precisely,
\[
\bar{\mathcal F}(\mu)
=
\inf_{n\geq1}
\left(
\frac1n\int_Xf_n\,d\mu+\frac Cn
\right)
=
\bar{\mathcal G}(\mu).
\]
We define the maximal ergodic value of the almost-additive
family by
\[
\beta(\mathcal F)
:=
\sup_{\mu\in\mathcal M_\sigma(X)}
\bar{\mathcal F}(\mu).
\]
In particular, $\beta(\mathcal F)=\beta(\mathcal G)$.

The following result is then an immediate consequence.

\begin{theorem}\label{thm:almost}
Let $X\subset\Sigma$ be a closed subshift, and let
$\mathcal F=(f_n)_{n\geq1}$ be an almost-additive family
with constant $C\geq0$, whose members are upper semicontinuous
and bounded above. Suppose that $f_1$ is coercive and that
\[
\beta(\mathcal F)>-\infty.
\]
Then there exists $\mu_\infty\in\mathcal M_\sigma(X)$ such that
\[
\bar{\mathcal F}(\mu_\infty)=\beta(\mathcal F).
\]
\end{theorem}

\begin{proof}
Consider the family $\mathcal G=(g_n)_{n\geq1}$ defined
by $g_n=f_n+C$. As verified above, $\mathcal G$ is
subadditive and
\[
\bar{\mathcal G}(\mu)=\bar{\mathcal F}(\mu)
\quad\text{for every }\mu\in\mathcal M_\sigma(X).
\]
Moreover, each $g_n$ is upper semicontinuous and bounded
above. Since $g_1=f_1+C$, the coercivity of $f_1$ implies
that of $g_1$. Finally,
\[
\beta(\mathcal G)=\beta(\mathcal F)>-\infty.
\]
Applying Theorem~\ref{thm:subadditive} to $\mathcal G$,
we obtain $\mu_\infty\in\mathcal M_\sigma(X)$ such that
\[
\bar{\mathcal F}(\mu_\infty)
=
\bar{\mathcal G}(\mu_\infty)
=
\beta(\mathcal G)
=
\beta(\mathcal F).
\]
\end{proof}

\begin{remark}
The family in Section \ref{sec:matrices} is, in general, neither additive nor
almost-additive. For instance, let $d=2$ and
\[
  A_0 := \begin{pmatrix} 0 & 1\\ 0 & 0\end{pmatrix},
  \qquad
  A_1 := \begin{pmatrix} 0 & 0\\ 1 & 0\end{pmatrix},
\]
with $\|\cdot\|$ the operator norm. Then $\|A_0\| = \|A_1\| = 1$, but
$A_0^2 = 0$. Hence, for $x\in[00]_\Sigma$, we get $f_1(x) + f_1(\sigma x) = 0$
while $f_2(x) = -\infty$, and no constant $C\ge0$ can satisfy the
almost-additivity condition of Definition~\ref{def:almost}. Such families are
covered by Theorem~\ref{thm:subadditive} but not by Theorem~\ref{thm:almost}.
\end{remark}

\end{document}